\documentclass{article}
\usepackage{amsfonts, amsmath, amssymb, amsthm, graphicx}

\newtheorem{theorem}{Theorem}

\newtheorem{proposition}[theorem]{Proposition}

\theoremstyle{definition}
\newtheorem{example}[theorem]{Example}
\newtheorem{definition}[theorem]{Definition}
\newtheorem{remark}[theorem]{Remark}
\newcommand{\diam}{\operatorname{diam}}
\newcommand{\keywords}{\noindent \textit{Keywords:} }
\newcommand{\MSC}{\noindent \textit{2010 MSC:} }
\begin{document}

\title{Dissipative flows, global attractors and shape theory \footnote{The authors are supported by MINECO (MTM2015- 63612-P). The first author is also supported by the FPI grant BES-2013-062675 and the mobility grant EEBB-I-16-11119.}}

\author{H\'ector Barge\footnote{Facultad de C. C. Matem\'aticas, Universidad Complutense de Madrid, Madrid 28040, Spain; email:hbarge@ucm.es}, Jos\'e M. R. Sanjurjo\footnote{Facultad de C. C. Matem\'aticas, Universidad Complutense de Madrid, Madrid 28040, Spain; email:jose\_sanjurjo@mat.ucm.es}}

\date{}

\maketitle

\begin{center}
{\em Dedicated to the memory of Sibe Marde\v si\'c.}
\end{center}

\begin{abstract}
In this paper we study continuous parametrized families of dissipative flows, which are those flows having a global attractor. The main motivation for this study comes from the observation that, in general, global attractors are not robust, in the sense that small
perturbations of the flow can destroy their globality. We give a necessary and sufficient condition for a global attractor to be continued to a global attractor. We also study, using shape theoretical methods and the Conley index, the bifurcation global to non-global.
\end{abstract}

\MSC{34D23, 54H20, 55P55, 34D10, 37C70, 37B30.}
\vspace{3pt}


\section*{Introduction}
This paper is devoted to the study of continuous parametrized families of dissipative flows $\varphi_\lambda:M\times\mathbb{R}\to M$, where $\lambda\in [0,1]$ and $M$ is a non-compact, locally compact metric space. 

Dissipative flows have been introduced by Levinson \cite{Lev} and they are a very important class of flows since, as it has been proved by Pliss \cite{Pli}, they agree with those flows having a global attractor. We provide a proof of this fact for the sake of completeness (see Proposition~\ref{gat}). Dissipative flows play a crucial role in mathematical ecology since they model the limitations of the environment (see for example \cite{ButWalt, Gar, Hof, Sanfine}). An interesting reference regarding dissipative flows is \cite{Hal}.

The main motivation
for this paper comes from the observation, pointed out in Example 5, that,
in general, global attractors are not robust, in the sense that small
perturbations of the flow can destroy their globality.

In section~\ref{sec:1} we recall the definition of dissipative flow and we prove, for the sake of completeness, that a flow is dissipative if and only if it possesses a global attractor (Proposition~\ref{gat}). We also see in Example~\ref{bounded} that a small perturbation of a global attractor can produce
a non-global attractor. Moreover, we see that in fact, it can have a bounded region of attraction. This fact motivates the definition of coercive family of dissipative flows. In addition, we introduce the concept of uniformly dissipative family of flows and we prove Theorem~\ref{undis} which establishes that for a family of dissipative flows a global attractor continues to global attractors if and only if the family is uniformly dissipative. We illustrate this result by showing that in the Lorenz equations the global attractor continues to global attractors. We finish this section by proving Theorem~\ref{coercive}, in which a topological study of coercive families of dissipative flows in $\mathbb{R}^n$ is made. In particular we see that if a parametrized family of dissipative flows is coercive then, given a continuation $K_\lambda$ of the global attractor $K_0$ of $\varphi_0$,  an isolated invariant compactum $C_\lambda\subset\mathbb{R}^n\setminus K_\lambda$ with the shape of $S^{n-1}$ which attracts uniformly all the points of the unbounded component of its complement and repels all the points of the bounded one which are not in $K_\lambda$ is created for all $\lambda>0$ small. Besides, we see that the diameter of $C_\lambda$ goes to infinity when $\lambda$ goes to zero. Moreover, we see that the existence of such a $C_\lambda$ for all $\lambda>0$ small is sufficient for the family to be coercive.

In section~\ref{sec:2} we study in all its generality the mechanism which produces the bifurcation global to non-global
in families of dissipative flows. With this aim we introduce the definition of polar family of dissipative flows and we see (Proposition~\ref{polar}) that a family of dissipative flows is polar if and only if every continuation $K_\lambda$ of the global attractor $K_0$ of $\varphi_{0}$ satisfies that $K_\lambda$ is not a global attractor for all $\lambda>0$ small. In addition we prove Theorem~\ref{general} which establishes that if a parametrized family of dissipative flows is polar then, given a continuation $K_\lambda$ of the global attractor $K_0$ of $\varphi_0$, for all $\lambda>0$ small the maximal invariant compactum $C_\lambda\subset\mathbb{R}^n\setminus K_\lambda$ is non-empty, isolated and has trivial cohomological Conley index.  

 Through the paper we often consider the Alexandrov compactification $M\cup\{\infty\}$ and we extend the family of flows $\varphi_\lambda:M\times\mathbb{R}\to M$ to a family of flows, which we also denote by $\varphi_\lambda$ , on $M\cup\{\infty\}$ by leaving $\infty$ fixed for all $\lambda\in[0,1]$.

A form of homotopy theory, namely \emph{shape theory}, which is the most {sui\-ta\-ble} for the study of global topological properties in dynamics, will be {occa\-sio\-na\-lly used}. Although a deep knowledge of shape theory is not necessary to understand the paper we {re\-co\-mmend} to the reader Borsuk's monography \cite{Bormono} and the books by Marde\v{s}i\'c and Segal \cite{MarSe} and Dydak and Segal \cite{DySe} for an exhaustive treatment of the subject and \cite{KapRod, RobSal,Rob2, Sanin, SanMul, Sanjuni, Sanjsta, SGRAC, GMRSNLA, GMRSo} for a concise introduction and some applications to dynamical systems. 

We make use of some notions of algebraic topology. Good references for this material are the books of Hatcher and Spanier \cite{Hat,Span}. We will use the notation $H_*$ and $H^*$ for the singular homology and cohomology with integer coefficients respectively. 

  The main reference for the elementary concepts of dynamical systems will be \cite{BhSz} but we also recommend \cite{Rob1, PaMe, Pil}. By the \emph{omega-limit} of a point $x$ we understand the set $\omega(x)=\bigcap_{t>0}\overline{x[t,\infty)}$ while the \emph{negative omega-limit} is the set $\omega^*(x)=\bigcap_{t<0}\overline{x(-\infty,t]}$. An invariant compactum $K$ is \emph{stable} if every neighborhood $U$ of $K$
contains a neighborhood $V$ of $K$ such that $V[0,\infty
)\subset U$. Similarly, $K$ is \textit{negatively stable} if every
neighborhood $U$ of $K$ contains a neighborhood $V$ of $K$ such that $V
(-\infty ,0]\subset U$. The compact invariant set $K$ is said to be \textit{%
attracting} provided that there exists a neighborhood $U$ of $K$ such that $%
\emptyset\neq\omega (x)\subset K$ for every $x\in U$ and \emph{repelling} if there
exists a neighborhood $U$ of $K$ such that $\emptyset\neq\omega ^{\ast }(x)\subset K$ for
every $x\in U$. An \textit{attractor }(or \emph{asymptotically stable} compactum)\textit{\ }is an attracting stable set and a \textit{repeller }is
a repelling\textit{\ }negatively stable set. We stress the fact that
stability (positive or negative) is required in the definition of attractor
or repeller. If $K$ is an attractor, its region (or basin) of attraction $%
\mathcal{A}(K)$ is the set of all points $x\in M$ such that $\emptyset\neq\omega (x)\subset
K $. It is well known that $\mathcal{A}(K)$ is an open invariant set. The region of repulsion of a repeller $K$, $\mathcal{R}(K)$ is defined in a dual way. If in particular $\mathcal{A}(K)$ is the whole phase space we say that $K$ is a \emph{global attractor}. 

 Notice that, since we will deal with parametrized families of flows we will use the notation $\omega_\lambda(x)$, $\omega^*_\lambda(x)$, $\mathcal{A}_\lambda(K)$ and $\mathcal{R}_\lambda(K)$ to denote the omega-limit, negative omega-limit, region of attraction and region of repulsion with respect to the flow $\varphi_\lambda$. 
 
There is a nice topological relation between an attractor $K$ and its region of attraction $\mathcal{A}(K)$ expressed in terms of shape theory. It establishes that the inclusion $i:K\hookrightarrow\mathcal{A}(K)$ a shape equivalence \cite[Theorem~3.6]{KapRod}. As a consequence, if $K$ is a global attractor of a flow in an Euclidean space it must have the shape of a point. In particular, it must be connected.

An important class of invariant compacta is the so-called \emph{isolated invariant sets} (see \cite{Con,ConEast,East} for details). These are compact invariant sets $K$ which possess an \emph{isolating neighborhood}, i.e. a compact neighborhood $N$  such that $K$ is the maximal invariant set in $N$.  

 A special kind of isolating neighborhoods will be useful in the sequel, the so-called \emph{isolating blocks}, which have good topological properties. More precisely, an isolating block $N$ is an isolating neighborhood such that there are compact sets $N^i,N^o\subset\partial N$, called the entrance and the exit sets, satisfying
\begin{enumerate}
\item $\partial N=N^i\cup N^o$;
\item for each $x\in N^i$ there exists $\varepsilon>0$ such that $x[-\varepsilon,0)\subset M-N$ and for each $x\in N^o$ there exists $\delta>0$ such that $x(0,\delta]\subset M-N$;
\item for each $x\in\partial N-N^i$ there exists $\varepsilon>0$ such that $x[-\varepsilon,0)\subset \mathring{N}$ and for every $x\in\partial N-N^o$ there exists $\delta>0$ such that $x(0,\delta]\subset\mathring{N}$.
\end{enumerate} 

These blocks form a neighborhood basis of $K$ in $M$. 

Let $K$ be an isolated invariant set. Its \emph{Conley index} $h(K)$ is defined as the pointed homotopy type of the topological space $(N/N^o,[N^o])$, where $N$ is an isolating block of $K$.  A weak version of the Conley index which will be useful for us is the \emph{cohomological index} defined as $CH^*(K)=H^*(h(K))$.  The \emph{homological index} $CH_*(K)$ is defined in an analogous fashion using homology. Notice that in fact, $CH^*(K)\cong H^*(N, N^o)$ and $CH_*(K)\cong H_*(N,N^o)$. Our main references for the Conley index theory are \cite{Con, ConZehn, Sal}.

In this paper the concept of continuation of isolated invariant sets plays a crucial role. Let $M$ be a locally compact metric space, and let $\varphi_\lambda:M\times\mathbb{R}\to  M$ be a parametrized family of flows (parametrized by $\lambda\in[0,1]$, the unit interval). The family $(K_\lambda)_{\lambda\in J}$,  where $J\subset[0,1]$ is a closed (non-degenerate) subinterval and, for each $\lambda\in J$, $K_\lambda$ is an isolated invariant set for $\varphi_\lambda$ is said to be a \emph{continuation} if for each $\lambda_0\in J$ and each $N_{\lambda_0}$ isolating neighborhood for $K_{\lambda_0}$, there exists $\delta>0$ such that $N_{\lambda_0}$ is an isolating neighborhood for $K_\lambda$ for every $\lambda\in (\lambda_0 -\delta, \lambda_0 + \delta)\cap J$. We say that the family $(K_\lambda)_{\lambda\in J}$ is a continuation of $K_{\lambda_0}$ for each $\lambda_0\in J$.

Notice that \cite[Lemma~6.1]{Sal} ensures that if $K_{\lambda_0}$ is an isolated invariant set for $\varphi_{\lambda_0}$, there always exists  a continuation $(K_\lambda)_{\lambda\in J_{\lambda_0}}$ of $K_{\lambda_0}$ for some closed (non-degenerate) subinterval $\lambda_0\in J_{\lambda_0}\subset[0,1]$.

There is a simpler definition of continuation based on \cite[Lemma 6.2]{Sal}. There, it is proved that if $\varphi_\lambda : M \times\mathbb{R}\to M$ is a parametrized family of flows and if $N_1$ and $N_2$ are isolating neighborhoods of the same isolated invariant set for $\varphi_{\lambda_0}$, then there exists $\delta>0$ such that $N_1$ and $N_2$ are isolating neighborhoods for $\varphi_\lambda$, for every $\lambda\in(\lambda_0-\delta,\lambda_0 +\delta)\cap[0,1]$, with the property that, for every $\lambda$, the isolated invariant subsets in $N_1$ and $N_2$ which have $N_1$ and $N_2$ as isolating neighborhoods agree.

Therefore, the family $(K_\lambda)_{\lambda\in J}$, with $K_\lambda$ an isolated invariant set for $\varphi_\lambda$, is a continuation if for every $\lambda_0\in J$ there are an isolating neighborhood $N_{\lambda_0}$ for $K_{\lambda_0}$ and a $\delta > 0$ such that $N_{\lambda_0}$ is an isolating neighborhood for $K_\lambda$, for every $\lambda\in(\lambda_0-\delta,\lambda_0 +\delta)\cap J$.

We will make use of the fact that if $(K_\lambda)_{\lambda\in J}$ is a continuation then, for each $\lambda_1,\lambda_2\in J$, the Conley indices $h(K_{\lambda_1})$ and $h(K_{\lambda_2})$ agree (see \cite[Corollary~6.8]{Sal}). A consequence of this fact is that if $K_{\lambda_0}$ is a non-empty attractor and $(K_\lambda)_{\lambda\in J}$ is a continuation of it, then $K_\lambda$ is non-empty for each $\lambda\in J$.

We are interested in continuations $(K_\lambda)_{\lambda\in J}$, with $0\in J$, where $K_0$ is a global attractor. Since $K_0$ is an attractor, using \cite[Theorem~4]{Sanjuni} it follows that there exists $0<\lambda_0\in J$ such that, for $\lambda<\lambda_0$, $K_\lambda$ is an attractor which has the shape of $K_0$. As a consequence, if the phase space is a Euclidean space, then for small values of $\lambda$, $K_\lambda$ has the shape of a point and, in particular, it is connected.

Notice that, since this should not lead to any confusion, sometimes we will only say that $K_\lambda$ is a continuation of $K_{\lambda_0}$ without specifying the subinterval $J\subset[0,1]$ to which the parameters belong.

\section{Continuation of global attractors}\label{sec:1}

In this section we study parametrized families of dissipative flows and continuations of global attractors. In particular, we will see that a small perturbation of a global attractor can be a non-global attractor since, in fact, it can have a bounded region of attraction. We will also give a necessary and sufficient condition for the family $(K_\lambda)_{\lambda\in [0,1]}$ where, for each $\lambda\in[0,1]$ $K_\lambda$ is the global attractor of $\varphi_\lambda$, to be a continuation of the global attractor $K_0$ of $\varphi_0$. Besides all this we give a descriptive characterization of coercive families of dissipative flows. 

\begin{definition}[Levinson 1944]
We say that a flow $\varphi:M\times\mathbb{R}\to M$ is dissipative provided that for each $x\in M$, $\omega(x)\neq\emptyset$ and the closure of the set
\[
\Omega(\varphi)=\bigcup_{x\in M}\omega(x)
\] 
is compact.
\end{definition}

\begin{proposition}[Pliss 1966] \label{gat}
A flow $\varphi:M\times\mathbb{R}\to M$ is dissipative if and only if there exists a global attractor.
\end{proposition}

\begin{proof}
If $\varphi$ has a global attractor $K$, then it is clear that $\Omega(\varphi)\subset K$ and hence its closure is compact ensuring the dissipativeness of $\varphi$.

Conversely, suppose that $\varphi$ is dissipative. We extend the flow to the Alexandrov compactification $M\cup\{\infty\}$ leaving $\infty$ fixed. Let $N$ be a compact neighborhood of $\overline{\Omega(\varphi)}$. Then, $U=(M\cup\{\infty\})\setminus N$ is a neighborhood of $\{\infty\}$ and, given any point $x\in M\setminus N=U\setminus\{\infty\}$, since $\emptyset\neq\omega(x)\subset N$, there exists a time $t_x>0$ such that $xt_x\notin U$. Hence, by \cite[Lemma~3.1]{Sal} $\{\infty\}$ is a repeller.  

Since $M\cup\{\infty\}$ is compact and $\{\infty\}$ is a repeller, there exists a non-empty attractor $K\subset M$ dual to $\{\infty\}$. It is clear that this attractor $K$ must be a global attractor.
\end{proof}

\begin{remark}
The dissipativeness of $\varphi$ is equivalent to $\{\infty\}$ being a repeller.
\end{remark}

An easy consequence of these considerations is the following proposition.

\begin{proposition}
Let $\varphi $ be a dissipative flow in $\mathbb{R}^{n}$ and $K$ an
attractor of $\varphi $. Then $K$ is global if and only if $\mathcal{A}(K)\setminus K$
does not contain bounded orbits.
\end{proposition}

Hence, if $K$ is an attractor of a dissipative flow, $\mathcal{A}(K)$ being
bounded is in sharpest contrast to $K$ being global.

\begin{example}\label{bounded}
Consider the family of ordinary differential equations defined on the plane in polar coordinates
\begin{equation}\label{eqn1}
 \begin{cases}
    \dot{r}=-r^3\left(\dfrac{1}{r}-\lambda\right)^2,\\
 \dot{\theta}=1
  \end{cases}\quad \lambda\in [0,1]
\end{equation}
\begin{figure}[h]
\centering
\includegraphics[scale=0.8]{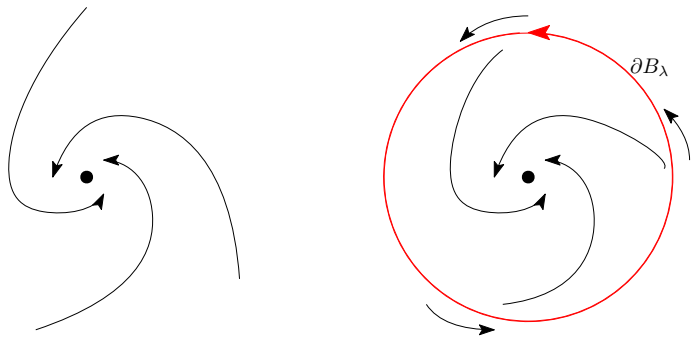} 
\caption{Phase portrait of the family of equations \eqref{eqn1}.}
\label{fig:1}
\end{figure}

The picture on the left in figure~\ref{fig:1} describes the phase portrait of equation \eqref{eqn1} when the parameter $\lambda=0$. We see that in this case the origin is a globally attracting fixed point and the orbit of any other point spirals towards it. The picture on the right describes the phase portrait of equation \eqref{eqn1} when $\lambda>0$. In this case we see that the origin is still an asymptotically stable fixed point but it is not a global attractor anymore since, for each $\lambda>0$, the circle centered at the origin and radius $1/\lambda$ is a periodic trajectory which attracts uniformly all the points of the unbounded component of its complement and repels all the points of the bounded one except the origin. Notice that for every $\lambda>0$ the disk $B_\lambda$ centered at the origin and radius $1/\lambda$  is a global attractor whose interior agrees with the basin of attraction of $\{0\}$. As a consequence of these remarks, for each $\lambda\in[0,1]$ the flow $\varphi_\lambda$ associated to the corresponding value of $\lambda$ in \eqref{eqn1} is dissipative and, if $K_\lambda$ is a continuation of the global attactor $\{0\}$ of $\varphi_0$, there exists $\lambda_0>0$ such that for each $0<\lambda<\lambda_0$, $K_\lambda=\{0\}$. Therefore, the global attractor $\{0\}$ cannot be continued to a family of global attractors.  
\end{example} 

Example~\ref{bounded} is an instance of parametrized family $\varphi _{\lambda }$ of dissipative flows in $\mathbb{%
R}^2$ with a global attractor $K_{0}$ such that every continuation $K_{\lambda }$ of $K_0$ satisfies that there exists $\lambda_0>0$ such that $\mathcal{A}_\lambda(K_{\lambda })$ is bounded for every $0<\lambda <\lambda_0.$ This motivates the following definition.

\begin{definition}
A parametrized family of dissipative flows $\varphi:M\times\mathbb{R}\to M$ is said to be \emph{coercive} if for any continuation $K_\lambda$ of the global attractor $K_{0}$ of $\varphi_0$,  there exists $\lambda_0>0$ such that  $\mathcal{A}_\lambda%
(K_{\lambda })$ is bounded for every $0<\lambda <\lambda _{0}$. 
\end{definition}

\begin{remark}
Notice that it follows from \cite[Lemma~6.2]{Sal} that if $K_0$ has a continuation $K_\lambda$ satisfying that there exists $\lambda_0>0$ such that  $\mathcal{A}_\lambda%
(K_{\lambda })$ is bounded for every $0<\lambda <\lambda _{0}$, then the family $\varphi_\lambda$ is coercive.
\end{remark}

Before studying coercive families of dissipative flows we will characterize when a global attractor can be continued to a family of global attractors. For this purpose we introduce the following definition.

\begin{definition}
A parametrized family of flows $\varphi_\lambda:M\times\mathbb{R}\to M$ is said to be uniformly dissipative provided that $\omega_\lambda(x)\neq\emptyset$ for each $x\in M$ and $\lambda\in[0,1]$ and the closure of
\[
\Omega=\bigcup_{\lambda\in[0,1]}\Omega(\varphi_\lambda)
\]
is compact.
\end{definition}

\begin{theorem}\label{undis}
Let $\varphi _{\lambda}:M\times\mathbb{R}\to M$ be a parametrized family of dissipative flows with $\lambda\in[0,1]$. Suppose that $K_\lambda$ denotes the global attractor of $\varphi_\lambda$. Then, the family $(K_\lambda)_{\lambda\in[0,1]}$ is a continuation of $K_0$ if and only if the family $(\varphi_\lambda)_{\lambda\in[0,1]}$ is uniformly dissipative.
\end{theorem}
\begin{proof}
Suppose that the family of global attractors $(K_\lambda)_{\lambda\in [0,1]}$ is a continuation of the global attractor $K_0$ of $\varphi_0$. Then, for each $\lambda_0\in[0,1]$ there exists an isolating neighborhood $N_{\lambda_0}$ of $K_{\lambda_0}$ and $\delta>0$ such that $N_{\lambda_0}$ is an isolating neighborhood of $K_\lambda$ for $\lambda\in (\lambda_0-\delta,\lambda_0+\delta)\cap [0,1]$. 

Using this fact and the compactness of $[0,1]$ it follows that there exist $\lambda_1,\ldots,\lambda_k\in[0,1]$ such that, if we call $K=\cup_{\lambda\in[0,1]}K_\lambda$, 
\[
\overline{K}\subset\bigcup_{i=1}^k N_{\lambda_k}.
\]
 As a consequence, $\overline{K}$ is compact. Moreover, $\omega_\lambda(x)\neq\emptyset$ for each $\lambda\in[0,1]$ and $x\in M$ since every flow in the family is dissipative and the closure of $\Omega$ is compact being contained in $\overline{K}$. Therefore, the family $(\varphi_\lambda)_{\lambda\in[0,1]}$ is uniformly dissipative. 
 
Conversely, suppose that the family $(\varphi_\lambda)_{\lambda\in[0,1]}$ is uniformly dissipative. We see that the family of global attractors $(K_\lambda)_{\lambda\in[0,1]}$ is a continuation of the global attractor $K_0$. For this purpose we consider the family of flows extended to the Alexandrov compactification $M\cup\{\infty\}$ leaving $\infty$ fixed. Let $\lambda_0\in[0,1]$, $N_{\lambda_0}$ an isolating neighborhood of $K_{\lambda_0}$ and $W$ a compact neighborhood of $N_{\lambda_0}\cup\Omega$. Then, $U=\overline{(M\cup\{\infty\})\setminus W}$ is an isolating neighborhood of $\{\infty\}$ for the flow $\varphi_{\lambda_0}$. Since $(K_{\lambda_0},\{\infty\})$ is an attractor-repeller decomposition, it follows from \cite[Lemma~6.4]{Sal} that there exists $\delta>0$ such that for each $\lambda\in (\lambda_0-\delta,\lambda_0+\delta)\cap[0,1]$ the pair $(A_\lambda,R_\lambda)$, where $A_\lambda$ is the isolated invariant set having $N_{\lambda_0}$ as isolating neighborhood and $R_{\lambda}$ is the maximal invariant set in $W$ for $\lambda\in(\lambda_0-\delta,\lambda_0+\delta)\cap[0,1]$, is an attractor-repeller decomposition. Besides, from the choice of $W$ it easily follows that $R_\lambda=\{\infty\}$ and, hence, $A_\lambda$ must agree with the global attractor $K_\lambda$ for $\lambda\in(\lambda_0-\delta,\lambda_0+\delta)\cap [0,1]$. Thus the result follows. 
\end{proof}

\begin{example}\label{lorenz}
An important example of a global attractor is provided by
the Lorenz equations 
\[
\begin{cases}
\dot{x}=\sigma (y-x) \\
\dot{y}=rx-y -xz \\
\dot{z}=xy-bz
\end{cases}
\]
where $\sigma ,r$ and $b$ are three real positive parameters. E.N. Lorenz
proved in \cite{Lor} that for every value of $\sigma ,r$ and $b$ there exists a global
attractor of zero volume for the flow associated to these equations. This
attractor should not be confused with the famous Lorenz attractor, which is
a proper subset of the global attractor.

If we fix $\sigma $ and $b$ we obtain a family of flows $\varphi _{r}:%
\mathbb{R}^{3}\times \mathbb{R\rightarrow R}^{3}$ corresponding to the
Lorenz equations for the different values of $r$. Each of these flows has a
global attractor $K_{r}$. We shall show that the family $\varphi _{r}$ is
uniformly dissipative and, hence, for a given $r_{0}$ the family of global
attractors $K_{r}$ is a continuation of $K_{r_{0}}$.

Sparrow studied in \cite{Spa} the function

\[
V=rx^{2}+\sigma y^{2}+\sigma (z-2r)^{2} 
\]%
which is a Lyapunov function for the flow $\varphi _{r}$. By using this
function he was able to prove that $K_{r}$ lies in a ball $B_{r}$ centered
at $0$ and with radius $O(r)$, such that $O(r)$ depends continuously on $r.$
Hence, if we consider an arbitrary $r_{0}$ and an interval $[c,d]$
containing $r_{0}$, we have that the set $C=\overline{\cup _{c\leq r\leq
d}B_{r}}$ is compact and that $\emptyset\neq\omega _{r}(x)\subset C$ for every $x\in 
\mathbb{R}^{3}$ and every $r$ with $c\leq r\leq d$. Hence the family of
Lorenz flows $\varphi _{r}$ is uniformly dissipative for any interval $[c,d]$
of parameters and the corresponding family $(K_{r})_{r\in[a,b]}$ of global attractors is
a continuation of $K_{r_{0}}$.
\end{example}

From now on we will focus on families of flows in $\mathbb{R}^n$.  An interesting fact, of topological and dynamical significance, takes place
for coercive families. This result provides a graphic characterization of
coercive families.

\begin{theorem}\label{coercive}
Let $\varphi _{\lambda }$, with $\lambda \in [0,1]$, be a a coercive
family of dissipative flows in $\mathbb{R}^{n}$. We denote by $K_{0}$ the
global attractor of $\varphi _{0}$ and by $K_{\lambda }$ a continuation of $K_0$.
Then there exists $\lambda _{0}>0$ such that for every $\lambda $ with $%
0<\lambda <\lambda _{0}$ there is an isolated invariant compactum $%
C_{\lambda }$ in $\mathbb{R}^{n}\setminus K_{\lambda }$ such that
\begin{enumerate}
\item[i)] $C_{\lambda }$ separates $\mathbb{R}^{n}$ into two components and $%
K_{\lambda }$ lies in the bounded component.

\item[ii)] $C_{\lambda }$ has the shape of $S^{n-1}$

\item[iii)] $C_{\lambda }$ attracts uniformly all the points of the unbounded component and
repels all the points of the bounded one which are not in $K_\lambda$.

\item[iv)] $\diam(C_{\lambda })\rightarrow \infty $ when $\lambda \rightarrow 0$, where $\diam(C_{\lambda})$ denotes the diameter of $C_\lambda$.
\end{enumerate}

Moreover, the existence of such a $C_{\lambda }$ for $0<\lambda<\lambda_0$ is sufficient for the
family to be coercive.
\end{theorem}

\begin{proof}
Suppose that $\varphi _{\lambda }$ is a coercive family of flows in $\mathbb{%
R}^{n}$. Then, if we extend this family to $\mathbb{R}^{n}\cup \{\infty \}$
leaving $\infty$ fixed, there are no connecting orbits between the attractor $K_{\lambda
}$ and $\{\infty\}$ when $0<\lambda <\lambda _{0}$ and, as a consequence, $%
\mathcal{A}_\lambda(K_{\lambda })$ and $\mathcal{R}_\lambda(\{\infty \})$ are disjoint open
sets. Besides, $\mathcal{R}_\lambda(\{\infty\})$ is connected by \cite[Theorem~3.6]{KapRod} and, since   there exists $0<\lambda_1\leq 1$ (we may assume that $\lambda_1=\lambda_0$) such that $K_\lambda$ is connected, \cite[Theorem~3.6]{KapRod} also ensures that $\mathcal{A}_\lambda(K_\lambda)$ is connected for $0<\lambda_1\leq 1$. Since $\mathbb{R}^{n}\cup \{\infty \}$ is connected, then
\[
C_{\lambda }=\mathbb{R}^{n}\setminus(\mathcal{A}_\lambda(K_{\lambda })\cup \mathcal{R}_\lambda%
(\{\infty \})) 
\]%
is non-empty, compact and invariant. It is clear that it is also isolated
since every point near $C_{\lambda }$ must be either in $\mathcal{A}_{\lambda}%
(K_{\lambda })$ or in $\mathcal{R}_\lambda(\{\infty \})$. Notice that from the
previous discussion it also follows that $C_{\lambda }$ decomposes $\mathbb{R%
}^{n}$ into two connected components and that $K_{\lambda }$ is contained in
the bounded one. Besides, it is clear that $C_{\lambda }$ must attract every
point in $\mathcal{R}_{\lambda}(\{\infty \})$ (the unbounded component) and repel
every point in $\mathcal{A}_{\lambda}(K_{\lambda })$ (the bounded one) which is not in $K_\lambda$. Now, we will
see that $\mathbb{R}^{n}\setminus K_{\lambda }$ has the homotopy type of $%
S^{n-1}$ and that the inclusion $i_\lambda:C_{\lambda }\hookrightarrow \mathbb{%
R}^{n}\setminus K_{\lambda }$ is a shape equivalence. To see that choose a closed
ball $B$ isolating $K_{0}$. Then, for small values of the parameter  $B$ is
an isolating neighborhood of $K_{\lambda }$ as well, and it is contained in $%
\mathcal{A}_\lambda(K_{\lambda })$ (in fact, we may assume that this happens for $0<\lambda<\lambda_0$). As a consequence $K_{\lambda }$ has a basis of
neighborhoods consisting of balls (obtained as images of $B$ by the flow $%
\varphi _{\lambda }$), i.e., it is cellular. Then, by \cite[Exercise~2, pg. 41]{Daver} $\mathbb{R}%
^{n}\setminus K_{\lambda }$ is homeomorphic to $\mathbb{R}^{n}\setminus\{p\}$, where $p\in \mathbb{R}^n$, and hence it has the homotopy type of $S^{n-1}$.

Let us see that $i_\lambda:C_{\lambda }\hookrightarrow \mathbb{R}^{n}\setminus K_{\lambda }$
is a shape equivalence. To do that notice that, since $K_{\lambda }$ is an
attractor for $0<\lambda<\lambda_0$, it posseses a Lyapunov function $\Phi_\lambda :\overline{\mathcal{A}_\lambda%
(K_{\lambda })}\rightarrow  [0,\infty ]$ which is strictly decreasing
on orbits of $\varphi_\lambda$ contained in $\mathcal{A}_\lambda(K_{\lambda })\setminus K_{\lambda }$ and such
that $\Phi_\lambda|_{K_{\lambda }}=0$ and $\Phi_\lambda|_{\partial \mathcal{A}_\lambda(K_{\lambda
})}=\infty $. Similarly, since $\{\infty \}$ is a repeller for every $\lambda$, there exist a
Lyapunov function $\Psi_\lambda :\overline{\mathcal{R}_\lambda(\{\infty \})}\rightarrow
[0,\infty ]$ such that $\Psi_\lambda $ is strictily increasing on orbits of $\varphi_\lambda$
contained in $\mathcal{R}_\lambda(\{\infty \})\setminus\{\infty \}$, $\Psi_\lambda (\infty )=0$ and $\Psi_\lambda |_{\partial \mathcal{R}_\lambda(\{\infty \})}=\infty $. Using this fact we define, for each $0<\lambda<\lambda_0$
a function $F_\lambda:\mathbb{R}^{n}\cup \{\infty \}\rightarrow  [0,\infty]$ by
\[
F_\lambda(x)=
\begin{cases}
\Phi_\lambda (x)&\mbox{if $x\in \mathcal{A}(K_{\lambda })$} \\
\Psi_\lambda (x)&\mbox{if $x\in \mathcal{R}(\{\infty \})$} \\
0&\mbox{if $x\in K_{\lambda }\cup \{\infty \}$} \\
\infty& \mbox{otherwise.}
\end{cases}
\]
It is clear that $F_\lambda$ is continuous for each $0<\lambda<\lambda_0$ and 
\[
C_{\lambda }=F_\lambda^{-1}(\infty ). 
\]

For every $a>0$ the set $N_{a}^\lambda=F_\lambda^{-1}[a,\infty ]$ is a compact neighborhood
of $C_{\lambda }$. On the other hand, for each $x\in \mathbb{R}%
^{n}\setminus(K_{\lambda }\cup \mathring{N}^\lambda_{a})$ there exists a unique time $%
t^\lambda_{a}(x)\in \mathbb{R}$ such that $F_\lambda(\varphi_\lambda(x,t^\lambda_{a}(x)))=a$ and for each $0<\lambda<\lambda_0$ the function $t^\lambda_{a}:\mathbb{R}^{n}\setminus(K_{\lambda }\cup \mathring{N}^\lambda_{a})\rightarrow \mathbb{R}$
is continuous. Using this fact we can define the retraction $r^\lambda_{a}:
\mathbb{R}^{n}\setminus K_{\lambda }\rightarrow N^\lambda_{a}$ as follows:
\[
r^\lambda_a(x)=
\begin{cases}
\varphi_\lambda(x,t^\lambda_{a}(x))&\mbox{if $x\in \mathbb{R}^{n}\setminus(K_{\lambda }\cup N^\lambda_{a})$}\\
x &\mbox{if $x\in N^\lambda_{a}$.}
\end{cases}
\]
It is straightforward to see that $r^\lambda_{a}$ is a deformation retraction
from $\mathbb{R}^{n}\setminus K_{\lambda }$ onto $N^\lambda_{a}$. Besides, if $a'>a$ it
is also easy to see that $r^\lambda_{a'}|_{N^\lambda_{a}}:N^\lambda_{a}\rightarrow N^\lambda_{a'
}$ is a deformation retraction as well. Since the family $\{N^\lambda_{a}\}_{a>0}$
is a basis of neighborhoods of $C_{\lambda }$ it follows that the inclusion $%
i_\lambda:C_{\lambda }\hookrightarrow \mathbb{R}^{n}\setminus K_{\lambda }$ is a shape
equivalence.

Moreover, for each $0<\lambda<\lambda_0$, the invariant compactum $C_\lambda$ contains the compact set $\partial\mathcal{A}_\lambda(K_\lambda)$ and, as a consequence, $\diam(C_\lambda)\geq\diam(\partial\mathcal{A}_\lambda(K_\lambda))=\diam(\overline{\mathcal{A}_\lambda(K_\lambda)})$. Besides, since $K_0$ is a global attractor it follows that $\diam(\overline{\mathcal{A}_\lambda(K_\lambda)})$ must go to infinity when $\lambda$ goes to zero. Therefore, $\diam(C_\lambda)\to\infty$  when $\lambda\to 0$. 

Conversely, assume now that there exists such a $C_{\lambda }$. Since $C_{\lambda }$
is an isolated invariant compactum which decomposes $\mathbb{R}^{n}$ into
two components with $K_{\lambda }$ lying in the bounded one it follows that $%
\mathcal{A}_\lambda(K_{\lambda })$ is also contained in the bounded component.

\end{proof}

\section{The general case}\label{sec:2}

In view of the previous results, it is interesting to study in all its
generality the mechanism which produces the bifurcation global to non-global
in families of dissipative flows. With this aim we introduce the following
definition.

\begin{definition}
Let $\varphi _{\lambda }:\mathbb{R}^{n}\times \mathbb{R\rightarrow }\mathbb{R%
}^{n},$ with $\lambda \in [0,1]$, be a parametrized family of
dissipative flows. The family is said to be \emph{polar} if it has arbitrarily
large bounded trajectories. More precisely: for every $L>0$ (arbitrarily
large) there is a $\lambda _{0}>0$ such that for every $0<\lambda <\lambda
_{0}$ there is a bounded trajectory $\gamma _{\lambda }$ of $\varphi
_{\lambda }$ and a $t_{\lambda }<0$ such that $\|\gamma _{\lambda }(t)\|>L$
for every $t$ with $-\infty <t<t_{\lambda }$.
\end{definition}

\begin{remark}\label{large}
Obviously, if $K_\lambda$ is a continuation of the global attractor $K_0$ of $\varphi_0$, for $L$ sufficiently large, $\gamma _{\lambda }$ lies in $\mathbb{%
R}^{n}\setminus K_{\lambda }$.
\end{remark}

The following proposition makes it clear that polarity is a key notion
regarding the transition global to non-global.

\begin{proposition}\label{polar}
Let $\varphi _{\lambda }:\mathbb{R}^{n}\times \mathbb{R\rightarrow }\mathbb{R%
}^{n},$ with $\lambda \in \lbrack 0,1]$, be a parametrized family of
dissipative flows. Then the family is polar if and only if for every continuation $K_\lambda$ of the global attractor $K_0$ of $\varphi_0$ there exists $%
\lambda _{0}$ such that $K_{\lambda }$ is not a global attractor for
every $0<\lambda <\lambda _{0}$.
\end{proposition}

\begin{proof}
If $\varphi _{\lambda }$ is polar then, by Remark~\ref{large}, for every $\lambda $
with $0<\lambda <\lambda _{0}$ there is a bounded trajectory $\gamma
_{\lambda }$ of $\varphi _{\lambda }$ contained in $\mathbb{R}%
^{n}\setminus K_{\lambda }$ and, thus, $K_{\lambda }$ is not a global attractor. Conversely, if $%
K_{\lambda }$ is not a global attractor for every $\lambda $ with $0<\lambda <\lambda
_{0}$, then there exists at least a bounded orbit $\gamma _{\lambda }$ in $%
\mathcal{A}_\lambda(K_{\lambda })\setminus K_{\lambda }$. The negative omega-limit of $\gamma
_{\lambda }$ must be contained in $\partial \mathcal{A}_\lambda(K_{\lambda })$,
which is a closed set all whose points are arbitrarily close to $\infty $ when $\lambda
\rightarrow 0$. This implies that the orbits $\gamma _{\lambda }$ satisfy
the requirements of the definition of polarity when $\lambda _{0}$ is
sufficiently small.
\end{proof}

The following result describes the overall picture of polar families of dissipative flows.

\begin{theorem}\label{general}
If $\varphi _{\lambda }:\mathbb{R}^{n}\times \mathbb{R\rightarrow }\mathbb{R}%
^{n},$ with $\lambda \in \lbrack 0,1]$, is a polar family of dissipative
flows then there exists $\lambda _{0}>0$ such that for every $\lambda $ with $%
0<\lambda <\lambda _{0}$ the maximal invariant compactum lying in $\mathbb{R}^{n}\setminus K_{\lambda }$ for the flow $\varphi _{\lambda }$, which we denote
by $C_{\lambda }$, is non-empty, isolated and its cohomological Conley index is trivial in every dimension.  Moreover, the family $\varphi _{\lambda }$ is
coercive if and only if $C_{\lambda }$ has the shape of $S^{n-1}.$
\end{theorem}

\begin{proof}
Let $K_{\lambda }$ be a continuation of the global attractor of $\varphi
_{0}$. Since $\varphi _{\lambda }$ is dissipative there exists a global attractor $A_{\lambda }$ of $\varphi
_{\lambda }$ and, since $\varphi _{\lambda }$ is polar, there exists $\lambda_0>0$ such that $K_\lambda$ is a connected attractor and $K_{\lambda}\varsubsetneq A_{\lambda }$ for $0<\lambda <\lambda_{0}$.

Since $K_{\lambda }$ is an attractor of $\varphi _{\lambda }$, it must be
also an attractor of $\varphi _{\lambda }|_{A_{\lambda }}$. Let $C_{\lambda }$
be the complementary repeller of $K_{\lambda }$ for $\varphi _{\lambda
}|_{A_{\lambda}}$. It is clear that $C_{\lambda }$ is a non-empty and
invariant compactum. Moreover, if $x$ is a point in $\mathbb{R}%
^{n}\setminus K_{\lambda }$ not belonging to $C_{\lambda }$ then the trajectory of $x$
either approaches $\infty $ for negative times (if $x\in \mathbb{R}%
^{n}\setminus A_{\lambda }$) or approaches $K_{\lambda }$ for positive times (if $%
x\in A_{\lambda }$). Both situations prevent $x$ from lying in a compact
invariant set contained in $\mathbb{R}^{n}\setminus K_{\lambda }$ and, thus, $%
C_{\lambda }$ is the maximal invariant compactum lying in $\mathbb{R}%
^{n}\setminus K_{\lambda }$. The same argument proves that $C_{\lambda }$ is isolated.

Now we will see that the cohomological Conley index $CH^{\ast }(C_{\lambda
})=0$. Consider the family of flows extended to $\mathbb{R}^{n}\cup \{\infty
\}$ leaving $\infty$ fixed. Let $N$ be a positively invariant isolating block of $K_{\lambda }$
such that $N\cap C_{\lambda }=\emptyset $ \ and let $N_{\infty }$ be a
negatively invariant isolating block of $\{\infty \}$ with $N_{\infty }\cap
C_{\lambda }=\emptyset $. Then, the compactum $\hat{N}=\mathbb{R}^{n}\setminus(%
\mathring{N}\cup \mathring{N}_{\infty })$ is an isolating block of $%
C_{\lambda }$ whose exit set $\hat{N}^{o}$ agrees with $\partial N$ and its
entrance set $\hat{N}^{i}$ agrees with $\partial N_{\infty }$. Moreover, it
is easy to see that $\hat{N}/\partial N$ and $\hat{N}/\partial N_{\infty }$
are connected. This implies that $CH^{0}(C_{\lambda })=0$ and that $%
CH_0^{-}(C_{\lambda })=0$, where $CH^{-}_{\ast }(C_{\lambda })$ is the
homological Conley index for the reverse flow. Then, by the time duality of the Conley index \cite{MrSr} we also have that $%
CH^{n}(C_{\lambda })=0$.

We study now the Morse-Conley sequence of cohomology of the attractor-
repeller decomposition $(K_{\lambda },C_{\lambda })$ of $A_{\lambda }$ \cite{Fr}
\[
\cdots\xrightarrow{\delta} CH^{k}(C_{\lambda
})\rightarrow CH^{k}(A_{\lambda })\rightarrow CH^{k}(K_{\lambda
})\xrightarrow{\delta}\cdots 
\]

Since $A_{\lambda }$ is a global attractor it follows from \cite[Lemma~2]{Sanjuni} that $CH^{k}(A_{\lambda
})=0$ for every $k>0$ and $CH^{0}(A_{\lambda })=\mathbb{Z}$. On the other hand, since $K_{\lambda }$ is a continuation of $%
A_{0} $ and the Conley index is preserved by continuation 
we also have that $CH^{k}(K_{\lambda })=0$ for every $k>0$ and that $%
CH^{0}(K_{\lambda })=\mathbb{Z}.$ We deduce from this that $%
CH^{k}(C_{\lambda })=0$ as well for $k>1.$ Since we already know that $%
CH^{0}(C_{\lambda })=CH^{n}(C_{\lambda })=0$ it remains only to prove that $%
CH^{1}(C_{\lambda })$ is also trivial.

Consider the initial segment of the Morse-Conley sequence 
\[
0\rightarrow CH^{0}(C_{\lambda })\rightarrow CH^{0}(A_{\lambda })\rightarrow
CH^{0}(K_{\lambda })\rightarrow CH^{1}(C_{\lambda })\rightarrow
CH^{1}(A_{\lambda })\rightarrow\cdots
\]

Then, the connectedness of $A_\lambda$ and $K_\lambda$ ensure that the
homomorphism
\[
\mathbb{Z}\cong CH^{0}(A_{\lambda })\rightarrow CH^{0}(K_{\lambda })\cong 
\mathbb{Z} 
\]%
is an isomorphism. Therefore, since $CH^1(A_\lambda)=0$, it follows that $CH^{1}(C_{\lambda })=0$.

The last part of the theorem follows easily from Theorem~\ref{coercive}.

\end{proof}

\begin{remark}
If we only assume that the family $\varphi _{\lambda }$ is not uniformly
dissipative then we have an analogous proposition but replacing ``there
exists $\lambda _{0}$ such that for every $0<\lambda <\lambda _{0}$" by ``for
every $\lambda _{0}$ there exists $0<\lambda <\lambda _{0}"$.
\end{remark}

\begin{center}
\textbf{Acknowledgements}
\end{center}
\vspace{.15cm}

Some of the results of this paper where obtained when the first author was visiting Jagiellonian University in Krakow. He wishes to express his gratitude to the Mathematics Department of Jagiellonian University and, very especially, to Klaudiusz W\'ojcik for his hospitality.

\bibliographystyle{amsplain}
\bibliography{Biblio1}
\end{document}